\documentclass[preprint,12pt]{elsarticle}
\usepackage{amssymb,amsthm,amsmath}

\newtheorem{theo}{Theorem}
\newtheorem{prop}{Proposition}
\newtheorem{lemma}{Lemma}
\newtheorem{remark}{Remark}

\makeatletter
\def\ps@pprintTitle{%
   \let\@oddhead\@empty
   \let\@evenhead\@empty
   \def\@oddfoot{\reset@font\hfil\thepage\hfil}
   \let\@evenfoot\@oddfoot
}
\makeatother

\begin{document}

\begin{frontmatter}
\title{On a Berry-Esseen type limit theorem for Boolean convolution}
\author[aut1]{Mauricio Salazar}
\ead{maurma@cimat.mx}
\address[aut1]{Instititu de F\'isica, UASLP}

\begin{abstract}
We obtain a sharp estimate of the speed of convergence in the Boolean central limit theorem for measures of finite sixth moment. The main tool is a quantitative version of the Stieltjes-Perron inversion formula. 
\end{abstract}


\end{frontmatter}

\section{Introduction}
In the previous work \cite{AS} of Arizmendi and the author, 
we proved that the speed of convergence in the Boolean central limit theorem is of order $O(\frac{1}{\sqrt{n}})$ for measures with bounded support and of order $O(\frac{1}{\sqrt[3]{n}})$ for measures with finite fourth moment. 
Also in \cite{AS} an example is given, consisting on an atomic measure with exactly two atoms, that shows that the rate $\frac{1}{\sqrt{n}}$ can not be improved. Thus, the estimate of the speed of convergence for measures of bounded support is sharp. 

In this paper we obtain an improvement of the above results by showing that the speed of convergence in the Boolean central limit theorem for measures of finite sixth moment is of order $O(\frac{1}{\sqrt{n}})$, and thus by the example mentioned above, this estimate is sharp. 


\begin{theo}\label{T1}
Let $\mu$ be a probability measure such that $m_1(\mu)=0$, $m_2(\mu)=1$, and $m_6(\mu)<\infty$. Define the measure $\mu_n:=D_{\frac{1}{\sqrt{n}}}\mu^{\uplus n}$, where $\uplus$ stands for the Boolean convolution. Then, for $n$ large enough we have that 
\begin{equation}\nonumber
    d_{lev}(\mu_n,\mathbf{b})\leq \frac{7}{2}\frac{C+2}{\sqrt{n}},
\end{equation}
where $\mathbf{b}$ denotes the symmetric Bernoulli distribution, $\frac{1}{2}\delta_{-1}+\frac{1}{2}\delta_{1},$ and $C$ is a constant that depends only on $\mu$.
\end{theo}

The proof of this theorem relies on refinements of some estimates related to the Cauchy transform of a measure. These precise estimates are given in Theorem 2, which provides  a quantitative version of the Stieltjes-Perron inversion formula. Also in Proposition 2 we show an asymptotic upper bound for the Cauchy transform of a measure in terms of its moments.  

The sections are organized as follows. In Section 2, we present the preliminary material and technical results necessary to prove our main result. Particularly, in Section 2.2, we present Theorem 2 and discuss some interesting consequences. In Section 3, we prove Theorems 1 and 2.

\section{Preliminaries}

\subsection{The L\'evy distance}

Let $\mu$ and $\nu$ probability measures. We define the \textbf{L\'evy distance} between them to be
$$d_{lev}(\mu,\nu):=inf \{ \epsilon>0 \: | \: F(x-\epsilon)-\epsilon\leq G(x)\leq F(x+\epsilon)+\epsilon \: \: \: \text{for all } x\in \mathbb{R} \},$$
where $F$ and $G$ are the cumulative distribution functions of $\mu$ and $\nu$ respectively.

The following Proposition is the Lemma 2 in \cite{AS}, and it is a key ingredient in the proof of Theorem 1.

\begin{prop}\label{ldmuab}
Let $\mu$ be a probability measure of zero mean and unit variance. Suppose further that $\mu((-1-\epsilon,-1+\epsilon)\cup (1-\epsilon,1+\epsilon))\geq 1-\epsilon$ for some $\epsilon\in (0,1)$. Then
\begin{equation}\nonumber
d_{lev}(\mu,\mathbf{b})\leq \frac{7}{2}\epsilon.
\end{equation}
\end{prop}

\subsection{The Cauchy Transform}


Throughout the paper $z$ denotes a complex number and we write $z=x+iy$, where $x$ and $y$ are real numbers.

The Cauchy transform (or Stieltjes transform) of a non-negative Borel measure $\mu$ is defined as
\begin{equation}\nonumber
    G_{\mu}(z):=\int_{\mathbb{R}}\frac{1}{z-t}d\mu(t) \quad \text{for }z\in \mathbb{C}^+,
\end{equation}
where $\mathbb{C}^+$ denotes the open upper complex half-plane.

We can recover a measure $\mu \in \mathcal{M}$ from its Cauchy transform via the Stieltjes-Perron inversion formula:

\begin{equation}\label{StieljesInversion}\nonumber
\mu([a,b])=\displaystyle\lim_{y \downarrow 0}-\frac{1}{\pi}\int_{a}^{b}Im(G_{\mu}(x+iy))dx,
\end{equation}
provided that $\mu(\{a,b\})=0$.

The following theorem is a quantitative version of the Stieltjes-Perron inversion formula which is tailored for our purposes. This will be proved in Section 3.1.

\begin{theo}\label{T2}
Let $\mu$ be a probability measure. Let $-\infty\leq a <b \leq \infty$. Then we have that for all $y>0$
\begin{equation}\nonumber
    \mu((a+\delta ,b-\delta])-\frac{2y}{\pi\delta}\leq -\frac{1}{\pi}\int_{a}^b ImG_{\mu}(x+iy)dx\leq \mu((a-\delta,b+\delta] ) +\frac{2y}{\pi \delta}.
\end{equation}
\end{theo}

Note that if $\delta=\sqrt{y}$ and $y\downarrow 0$, then we obtain the Stieltjes-Perron inversion formula for probability measures.
Moreover, taking $a=-\infty$ and $\delta=\sqrt{\frac{2y}{\pi}}$,  we deduce that

\begin{equation}\nonumber
    d_{lev}(\mu^{y},\mu)\leq \sqrt{\frac{2y}{\pi}},
\end{equation}
where $\mu^{y}$ is the probability measure of density $-\frac{1}{\pi}Im(G_{\mu}(x+iy))dx$, i.e. $\mu^{y}=\mu*C_y$ where $C_y$ is the Cauchy distribution with location 0 and scale $y$.
This further implies the following inequality for the L\'evy distance between two probability measures in terms of the Cauchy transform
\begin{equation}\nonumber
    d_{lev}(\mu,\nu)\leq \sqrt{\frac{8y}{\pi}}+\frac{1}{\pi}\int_{\mathbb{R}} |Im(G_{\mu}(z))-Im(G_{\nu}(z))|dx.
\end{equation}

Next, we discuss some bounds for the Cauchy transform.

Since $|z-t|\geq y$ for all $t\in \mathbb{R}$, then it follows that
\begin{equation} \label{equ.cauchy.bound} 
    |G_{\mu}(z)|\leq \frac{\mu(\mathbb{R})}{y} \quad \text{for }z\in \mathbb{C}^+. 
\end{equation}

The following proposition gives another bound for $|G_{\mu}(z)|$, where instead we use the real part of $z$. In practice, this is very useful since $x$ is typically much larger than $y$. 
\begin{prop}\label{pro-cau.tra.est}
Let $\mu$ be a measure and $i\geq 0$. Then we have that
   \begin{equation}\nonumber
       |G_{\mu}(z)|< \frac{2\mu(\mathbb{R})}{|x|}+\frac{2^{i}\int_{\mathbb{R}}|t|^i d\mu(t)}{y|x|^i}, 
   \end{equation}
for any $x>0$.
\end{prop}

\begin{proof}
We have that
\begin{align}
   |G_{\mu}(z)| &\leq \int_{|t|\leq \frac{|x|}{2}} \frac{1}{|x-t|}d\mu(t)+\int_{|t|>\frac{|x|}{2}} \frac{1}{y}d\mu(t) \nonumber \\
   &\leq \frac{2}{|x|}\int d\mu(t)+\frac{1}{y} \int_{|t|>\frac{|x|}{2}}d\mu(t) \nonumber \\
   &< \frac{2\mu(\mathbb{R})}{|x|}+\frac{2^{i}\int_{\mathbb{R}}|t|^i d\mu(t)}{y|x|^i}.
   \nonumber
\end{align}
\end{proof}

In particular, taking $i=2$ we obtain that
\begin{equation}\label{eq.abs.G.0} 
 |G_{\mu}(z)| < \frac{2\mu(\mathbb{R})}{|x|}+\frac{ 4\int t^2 d\mu(t)}{yx^2}.  
\end{equation}




The \textbf{reciprocal Cauchy transform} (or $F$-transform) of a positive Borel measure $\mu \in \mathcal{M}$ is defined as
\begin{equation}\nonumber
F_{\mu}(z):=\frac{1}{G_{\mu}(z)}\quad \:\text{for } z\in \mathbb{C}^+.    
\end{equation}

Directly by definition, it is not too difficult to see that for a probability measure $\mu$ and $a>0$, then
\begin{equation}\label{eq.Ftrans.dil}
F_{D_a \mu}(z)=aF_{\mu}(z/a) \quad \text{for } z\in \mathbb{C}^+,
\end{equation}
where $D_a\mu$ denotes the dilation of a measure $\mu$ by a factor $a>0$; this means that $D_a\mu(B)=\mu(a^{-1}B)$ for all Borel sets $B \subset \mathbb{R}$.

The next proposition gives a fundamental representation of the $F$-transform for the probability measures that are of our interest.

\begin{prop}\label{prop.Ftrans-represent}
Let $\mu$ a probability measure such that $m_1(\mu)=0$, $m_2(\mu)=1$, and $m_6(\mu)<\infty$. Then there exists a real number $\alpha$ and a non-negative Borel measure $\omega$ such that $m_2(\omega)<\infty$ and

\begin{equation}\nonumber
    F_{\mu}(z)=z-\frac{1}{z-\alpha-G_{\omega}(z)} \quad \text{for } z\in \mathbb{C}^+. 
\end{equation}
\end{prop}

\begin{proof}
Let $\mu$ as in the hypothesis. Since $m_1(\mu)=0$ and $m_2(\mu)=1$, then by Proposition 2.1 in \cite{Has} we have that there exists a probability measure $\nu$ such that
\begin{equation}\nonumber
    F_{\mu}(z)=z-G_{\nu}(z) \quad \text{for } z\in \mathbb{C}^+.
\end{equation}

Moreover, since $m_6(\mu)<\infty$, then by Proposition 4.8 also in \cite{Has} we have that $m_4(\nu)<\infty$. Again by the same propositions aplied to $\nu$, we get that there exists a real number $\alpha$ and non-negative measure $\omega$ such that $m_2(\omega)<\infty$ and 
\begin{equation}\nonumber 
    F_{\nu}(z)=z-\alpha-G_{\omega}(z) \quad \text{for } z\in \mathbb{C}^+.
\end{equation}

The desired representation follows from the above equations.

\end{proof}

\begin{remark}
It can be shown that $\alpha=m_3(\mu)$ and $\omega(\mathbb{R})=m_4(\mu)-m_3(\mu)^2-1$ by using in the previous proof that $G_{\nu}(z)=z-F_{\mu}(z)$ is the K-transform of $\mu$, see \cite{SW}, and then expressing the first moments of $\nu$ in terms of the Boolean cumulants of $\mu$, see \cite{AS}, which in turn would be in terms of the first moments of $\mu$.
\end{remark}

\subsection{Boolean convolution}

Given probability measures $\mu$ and $\nu$, the \emph{Boolean convolution} $\mu \uplus \nu$, introduced by Speicher and Woroudi \cite{SW}, is the probability measure defined by the equation

\begin{equation}\nonumber
F_{\mu \uplus \nu}(z)=F_{\mu}(z)+F_{\nu}(z)-z \quad \text{for} \: z\in \mathbb{C}^+.    
\end{equation}

Let $\mu$ be a probability measure and $n$ be a positive integer. We want to obtain an expression for the F-transform of $\mu_n:=D_{\frac{1}{\sqrt{n}}}(\mu^{\uplus n})$.

First note that
\begin{equation}\nonumber
F_{\mu^{\uplus n}}(z)=(1-n)z+nF_{\mu}(z) \quad \text{for } z\in \mathbb{C}^+.
\end{equation}

Now, suppose further that $m_1(\mu)=0$, $m_2(\mu)=1$, and $m_6(\mu)<\infty$. Thus, by Proposition $\ref{prop.Ftrans-represent}$ we get that

\begin{equation}\nonumber
F_{\mu^{\uplus n}}(z)=z-\frac{n}{z-\alpha-G_{\omega}(z)} \quad \text{for } z\in \mathbb{C}^+,
\end{equation}
where $\alpha$ is a real number and and $\omega$ is a non-negative Borel measure such that $\int t^2 d\omega(t)<\infty$. Finally, applying  ($\ref{eq.Ftrans.dil}$) we obtain the representation

\begin{equation}\label{eq.Fun}
    F_{\mu_n}(z)=z-\frac{1}{z-\frac{\alpha}{\sqrt{n}}-\frac{1}{\sqrt{n}}G_{\omega}(\sqrt{n}z)} \quad \text{for } z\in \mathbb{C}^+. 
\end{equation}

\section{Proofs} 

\subsection{Proof of Theorem \ref{T2}.} 
 
Let $\mu$ be a probability measure. Choose $a$ and $b$ such that $-\infty\leq a <b \leq \infty$ and fix $y>0$. First, we rewrite the following integral

\begin{align}
    \int_{a}^b -ImG_{\mu}(x+iy)dx&=\int_{a}^b\int_{-\infty}^{\infty} \frac{y}{(x-t)^2+y^2}d\mu(t)dx \nonumber \\
    &=\int_{-\infty}^{\infty}\int_{a}^b \frac{y}{(x-t)^2+y^2}dxd\mu(t)
    \nonumber \\
    &=\int_{-\infty}^{\infty}\int_{\frac{a-t}{y}}^{\frac{b-t}{y}} \frac{1}{x^2+1}dxd\mu(t). \nonumber
\end{align}

Now, let $\delta \in (0,\frac{b-a}{2})$. It follows that

\begin{align}
    \int_{-\infty}^{\infty}\int_{\frac{a-t}{y}}^{\frac{b-t}{y}} \frac{1}{x^2+1}dxd\mu(t) &\geq \int_{a+\delta}^{b-\delta}\int_{\frac{a-t}{y}}^{\frac{b-t}{y}} \frac{1}{x^2+1}dxd\mu(t) \nonumber \\
    &\geq \int_{a+\delta}^{b-\delta}\int_{\frac{-\delta}{y}}^{\frac{\delta}{y}} \frac{1}{x^2+1}dxd\mu(t) \nonumber \\
    &= \int_{a+\delta}^{b-\delta}\big(\pi-\int_{|t|>\frac{\delta}{y}} \frac{1}{x^2+1}dx\big)d\mu(t) \nonumber\\
    &\geq \int_{a+\delta}^{b-\delta}(\pi-\frac{2y}{\delta})d\mu(t) \nonumber\\
    &\geq \pi\mu((a+\delta,b-\delta])-\frac{2y}{\delta}. \nonumber
\end{align}

So, we arrive to
\begin{equation}\nonumber
     \mu((a+\delta,b-\delta])-\frac{2y}{\pi\delta}\leq -\frac{1}{\pi}\int_{a}^b ImG_{\mu}(x+iy)dx.
\end{equation}
On the other hand, we have that
\begin{equation}\nonumber
    \int_{a}^b -ImG_{\mu}(x+iy)dx = \int_{a-\delta}^{b+\delta}\int_{\frac{a-t}{y}}^{\frac{b-t}{y}} \frac{1}{x^2+1}dxd\mu(t)+\int_{t\notin (a-\delta,b+\delta] }\int_{\frac{a-t}{y}}^{\frac{b-t}{y}} \frac{1}{x^2+1}dxd\mu(t).
\end{equation}

Next, note that

\begin{equation}\nonumber
    \int_{a-\delta}^{b+\delta}\int_{\frac{a-t}{y}}^{\frac{b-t}{y}} \frac{1}{x^2+1}dxd\mu(t)\leq \int_{a-\delta}^{b+\delta}\int_{-\infty}^{\infty} \frac{1}{x^2+1}dxd\mu(t)=\pi\mu((a-\delta,b+\delta]).
\end{equation}

Hence, splitting the integral over the complement of the interval $(a-\delta,b+\delta]$, we get
\begin{align}
    \int_{t\notin (a-\delta,b+\delta] }\int_{\frac{a-t}{y}}^{\frac{b-t}{y}} \frac{1}{x^2+1}dxd\mu(t)&=\int_{-\infty}^{a-\delta} \int_{\frac{a-t}{y}}^{\frac{b-t}{y}} \frac{1}{x^2+1}dxd\mu(t)+\int_{b+\delta}^{\infty} \int_{\frac{a-t}{y}}^{\frac{b-t}{y}} \frac{1}{x^2+1}dxd\mu(t)
    \nonumber \\
    &\leq\int_{-\infty}^{a-\delta} \int_{\frac{\delta}{y}}^{\infty} \frac{1}{x^2+1}dxd\mu(t)+\int_{b+\delta}^{\infty} \int_{-\infty}^{\frac{-\delta}{y}} \frac{1}{x^2+1}dxd\mu(t) \nonumber \\
    &\leq2\int_{\frac{\delta}{y}}^{\infty} \frac{1}{x^2+1}dxd\mu(t) \nonumber \\
    &=\frac{2y}{\delta}. \nonumber
\end{align}

Finally, we conclude that for any $y>0$
\begin{equation}\nonumber
    \frac{-1}{\pi}\int_{a}^b ImG_{\mu}(x+iy)dx\leq \mu((a-\delta,b+\delta]) +\frac{2y}{\pi\delta}.
\end{equation}

\subsection{Proof of Theorem \ref{T1}}

Fix a probability measure $\mu$ such that $m_1(\mu)=0$, $m_2(\mu)=1$, and $m_6(\mu)<\infty$. Define $\mu_n:=D_{\frac{1}{\sqrt{n}}}(\mu^{\uplus n})$. We begin by obtaining some representations for the imaginary part of the Cauchy transform of $\mu_n$.

Note that
\begin{equation}\label{eq.im.cauchy.un}
    -Im(G_{\mu_n}(z))=\frac{Im(F_{\mu_n}(Z))}{|F_{\mu_n}(z)|^2}.
\end{equation}
Recall that by ($\ref{eq.Fun}$) we have the representation
\begin{equation}\nonumber
    F_{\mu_n}(z)=z-\frac{1}{z-\frac{\alpha}{\sqrt{n}}-\frac{1}{\sqrt{n}}G_{\omega}(\sqrt{n}z)} \quad \text{for } z\in \mathbb{C}^+,
\end{equation}
where $\alpha$ is a real number and $\omega$ is a non-negative Borel measure such that $\int t^2 d\omega(t)<\infty$. Define $W_n(z)=z-\frac{\alpha}{\sqrt{n}}-\frac{1}{\sqrt{n}} G_{\omega}(\sqrt{n}z)$ so that $F_{\mu_n}(z)=z-\frac{1}{W_n(Z)}$.
It follows that
\begin{equation}\label{eq.im.cauchy2}
    -ImG_{\mu_n}(z)=\frac{y|W_n(z)|^2+Im(W_n(z))}{|zW_n(z)-1|^2}.
\end{equation}

Next, we establish two lemmas that carry the main estimations of the proof. But first, we define some constants and give an inequality that is vital for making such estimations.
Let $K=\max\{ \omega(\mathbb{R}), \int t^2 d\omega(t)\}$. Now,
take $C>\max \{ 5, \: |\alpha|+2,\: 4(K+1)^2,\: 1+\frac{1}{0.3^2}(30K+1)\}$, let $n>\max \{20^2\alpha^2,\: 20\cdot 30K,\: 16C^2 \}$, and fix $y=\frac{1}{n}$. 
Observe that by the inequality ($\ref{eq.abs.G.0}$) we deduce that 
\begin{equation}\label{eq.abs.G}
    |G_{\omega}(\sqrt{n}z)|<\frac{2K}{\sqrt{n}|x|}+\frac{4K}{\sqrt{n}x^2}.
\end{equation}
\begin{lemma}
 We have that $-\frac{1}{\pi}\int_{A_i} ImG_{\mu_n}(z)dx\leq \frac{1}{\pi \sqrt{n}}$ for $i=1,2,$ provided that $n$ is large enough, and where $A_1=(-\infty,-1-\frac{C}{\sqrt{n}}]$ and $A_2=[1+\frac{C}{\sqrt{n}},\infty)$.
\end{lemma}

\begin{proof}


Assume that $x \leq -1-\frac{C}{\sqrt{n}}$. Since $|x|>1$, then by ($\ref{eq.abs.G}$) we have that $|G_{\omega}(\sqrt{n}z)|< \frac{6K}{\sqrt{n}}$.

First, we want to bound below $|F_{\mu_n}(z)|$.  Observe that $Re(W_n(z))\leq  x+\frac{|\alpha|}{\sqrt{n}}+\frac{1}{\sqrt{n}}|G_{\omega}(\sqrt{n}z)|< x+\frac{|\alpha|}{\sqrt{n}}+\frac{6K}{n}$. As $x\leq -1-\frac{C}{\sqrt{n}}$, $\sqrt{n}>6K$, and $C> |\alpha|+1$, then it follows that $Re(W_n(z))<-1$. Therefore, $|W_n(z)|>1$, which further implies  $|\frac{1}{W_n(z)}|<1$. Hence, we deduce that $|Re(\frac{1}{W_n(z)})|<1$. Using this, we conclude that
\begin{equation}\nonumber
   |Re(F_{\mu_n}(z))|=|x-Re(\frac{1}{W_n(z)})|\geq |x|-|Re(\frac{1}{W_n(z)})|> -x-1,
\end{equation}
for $x \leq -1-\frac{C}{\sqrt{n}}$.

Now, we want to bound above $Im(F_{\mu_n}(z))$. As seen above, for $x\leq -1-\frac{C}{\sqrt{n}}$, one has that  $-Im(G_{\omega}(\sqrt{n}z))\leq |G_{\omega}(\sqrt{n}z)|< \frac{6K}{\sqrt{n}}$ and $|W_n(z)|>1$. Therefore, $Im(W_n(z))=y-\frac{1}{\sqrt{n}}Im(G_{\omega}(\sqrt{n}z))<\frac{1}{n} +\frac{6K}{n}$. Hence, we obtain that

\begin{equation}\nonumber
   Im(F_{\mu_n}(z))=y+\frac{Im(W_n(z))}{|W_n(z)|^2}< \frac{6K+2}{n}< \frac{C}{n}. 
\end{equation}

By the previous estimations and $(\ref{eq.im.cauchy.un})$, we conclude that $-Im(G_{\mu_n}(z))< \frac{C/n}{(x+1)^2}$ for $x\leq -1-\frac{C}{\sqrt{n}}$. It follows that
\begin{equation}\nonumber
    -\frac{1}{\pi}\int_{A_1} ImG_{\mu_n}(z)dx<\frac{1}{\pi}\int_{-\infty}^{-1-\frac{C}{\sqrt{n}}}\frac{C/n}{(x+1)^2}dx=\frac{1}{\pi \sqrt{n}}.
\end{equation}

The same estimation for $A_2$ follows from a similar argument.
\end{proof}

\begin{lemma}
 We have that $-\frac{1}{\pi}\int_{[-1+\frac{C}{\sqrt{n}},\:1-\frac{C}{\sqrt{n}}]} ImG_{\mu_n}(z)dx< \frac{2C}{3 \sqrt{n}}+\frac{6}{\pi \sqrt{n}}$.
\end{lemma}

\begin{proof}
We deliver the estimation of this integral in three parts.

First, let us suppose that $x\in [0.4,1-\frac{C}{\sqrt{n}}]$. By ($\ref{eq.abs.G}$), it follows  that $|G_{\omega}(\sqrt{n}z)| < \frac{30K}{\sqrt{n}}$. 

Our objective is to bound $-Im(G_{\mu_n}(z))$. We begin by bounding $Im(F_{\mu_n}(z))$.
We claim that $|W_n(z)|> 0.3$. Indeed, from the definition of $W_n(z)$ we see that 
$|W_n(z)|\geq |z|-\frac{|\alpha|}{\sqrt{n}}-|\frac{1}{\sqrt{n}}G_{\omega}(\sqrt{n}z)|>x-\frac{|\alpha|}{\sqrt{n}}-\frac{30K}{n}$. So, the claim follows as $n$ is larger than $20\cdot30K$ and $20^2\alpha^2$.

Next, note that $Im(W_n(z))=y-\frac{1}{\sqrt{n}}Im(G_{\omega}(\sqrt{n}z))< \frac{30K+1}{n} $. It follows that

\begin{equation}\nonumber
    Im(F_{\mu_n}(z))=y+\frac{Im(W_n(z))}{|W_n(z)|^2}< \frac{1+\frac{1}{0.3^2}(30K+1)}{n}<\frac{C}{n}.
\end{equation}

Now, let us bound below $|F_{\mu_n}(z)|$. 
Observe that
\begin{equation*}
   Re(\frac{1}{W_n(z)})=\frac{1+Im(W_n(z))Im(\frac{1}{W_n(z)})}{Re(W_n(z))}.
\end{equation*}
We have that $Re(W_n(z))\leq x+|\frac{\alpha}{\sqrt{n}} +\frac{1}{\sqrt{n}}G_{\omega}(\sqrt{n}z)|< 1-\frac{C}{\sqrt{n}}+\frac{|\alpha|}{\sqrt{n}}+\frac{30K}{n}$. Recall that $|W_n(z)|>0.3$ and $Im(W_n(z))< \frac{30K+1}{n}$. It follows that $1+Im(W_n(z))Im(\frac{1}{W_n(z)})>1-\frac{(30K+1)^2}{0.3^2n^2}$. Since the last quantity is bigger than $Re(W_n(z))$, as $C>|\alpha|+2$ and  $n$ is larger than $(30k+1)^2$ and $16$, then 
we deduce that $Re(\frac{1}{W_n(z)})>1$. Thus, 
\begin{equation}\nonumber
 |F_{\mu_n}(z)|\geq |x-Re(\frac{1}{W_n(z)})| \geq|Re(\frac{1}{W_n(z)})|-x> 1-x.
\end{equation}

By the above estimations and ($\ref{eq.im.cauchy.un}$), we obtain that $-Im(G_{\mu_n}(z))< \frac{C/n}{(x+1)^2}$ for $x\in  (0.4,1-\frac{C}{\sqrt{n}}]$. We conclude that

\begin{equation}\label{eq.int.media.1}
    -\frac{1}{\pi}\int_{0.4}^{1-\frac{C}{\sqrt{n}}} ImG_{\mu_n}(z)dx<\frac{1}{\pi}\int_{0.4}^{1-\frac{C}{\sqrt{n}}}\frac{C/n}{(1-x)^2}dx< \frac{1}{\pi \sqrt{n}}.
\end{equation}

With minor modifications on this argument, we can also conclude the same estimation of this integral for $x\in [-1+\frac{C}{\sqrt{n}},-0.4]$. 

Secondly, suppose that $x\in (\frac{\sqrt{C}}{\sqrt{n}},0.4]$. 
By ($\ref{eq.abs.G}$), it follows that $|G_{\omega}(\sqrt{n}z)|<(\frac{2}{\sqrt{C}}+\frac{2\sqrt{n}}{C})K$. 
Our goal is to bound the expression ($\ref{eq.im.cauchy2}$).

Note that $|W_n(z)|\leq |z|+\frac{|\alpha|}{\sqrt{n}}+|\frac{1}{\sqrt{n}}G_{\omega}(\sqrt{n}z)|<x+y+\frac{|\alpha|}{\sqrt{n}}+\frac{2K}{\sqrt{Cn}}+\frac{2K}{C}<1$ since $C>6K+1$ and $n$ is larger $20^2\alpha^2$ and $200K$. 
Moreover, as $|z|<\frac{1}{2}$, we obtain that

\begin{equation}\nonumber
 |zW_n(z)-1|\geq 1-|zW_n(z)|>\frac{1}{2}.   
\end{equation}

Now, we have that $Im(W_n(z))=y-\frac{1}{\sqrt{n}}Im(G_{\omega}(\sqrt{n}z))<\frac{1}{n}+\frac{1}{\sqrt{n}}(\frac{2}{\sqrt{n}x}+\frac{2}{\sqrt{n}x^2})K $, thus
\begin{equation}\nonumber
    y|W_n(z)|^2 +Im(W_n(z))\leq \frac{2}{n}+\frac{2K}{nx}+\frac{2K}{nx^2}. 
\end{equation}

By the above estimations and ($\ref{eq.im.cauchy2}$), we deduce that
\begin{equation}\nonumber
    -\frac{1}{\pi}\int_{\frac{\sqrt{C}}{\sqrt{n}}}^{0.4} ImG_{\mu_n}(z)dx< \frac{1}{\pi}\int_{\frac{\sqrt{C}}{\sqrt{n}}}^{0.4} \frac{\frac{2}{n}+\frac{2K}{nx}+\frac{2K}{nx^2}}{\frac{1}{2}}dx.
\end{equation}

Since
$\int_{\frac{\sqrt{C}}{\sqrt{n}}}^{0.4}\frac{1}{nx^2}<\frac{1}{\sqrt{C}\sqrt{n}}$ and $x>\frac{\sqrt{C}}{\sqrt{n}}$, then we conclude that

\begin{equation}\label{eq.int.media.2}
    -\frac{1}{\pi}\int_{\frac{\sqrt{C}}{\sqrt{n}}}^{0.4} ImG_{\mu_n}(z)dx< \frac{2}{\pi n}+\frac{6K}{\pi \sqrt{C}\sqrt{n}}<\frac{2}{\pi\sqrt{n}}.
\end{equation}

By a similar argument we can obtain the same estimation of this integral for $x\in [-0.4,-\frac{\sqrt{C}}{\sqrt{n}}]$.

Finally, suppose that $x\in [-\frac{\sqrt{C}}{\sqrt{n}},\frac{\sqrt{C}}{\sqrt{n}}]$. It follows that $|W_n(z)|\leq |z|+\frac{|\alpha|}{\sqrt{n}}+|\frac{1}{\sqrt{n}}G_{\omega}(\sqrt{n}z)| $. Moreover, since $|G_{\omega}(\sqrt{n}z)|<\frac{\omega(\mathbb{R})}{\sqrt{n}y}$, then we conclude that $|W_n(z)| < x+y+\frac{|\alpha|}{\sqrt{n}}+\omega(\mathbb{R})<\frac{\sqrt{C}}{2}$,
as $C>4(K+1)^2$ and $n$ is larger than $20^2\alpha^2$ and $16C^2$.
Next, note that $|zW_n(z)|\leq |z||W_n(z)|<(\frac{\sqrt{C}}{\sqrt{n}}+\frac{1}{n})\frac{\sqrt{C}}{2}<\frac{1}{2}$ as $n>16C^2$. Thus, we obtain that
\begin{equation*}
   |zW_n(z)-1|\geq 1-|zW_n(z)|>\frac{1}{2}. 
\end{equation*}

Now, by the above estimations, we also have that
\begin{equation}\nonumber
    y|W_n(z)|^2 +Im(W_n(z))< \frac{C}{4n}+\frac{\sqrt{C}}{2}. 
\end{equation}

By ($\ref{eq.im.cauchy2}$), we deduce that $-Im(G_{\mu_n}(z))<\frac{C}{2n}+\sqrt{C}$, and so we conclude that

\begin{equation}\nonumber
    -\frac{1}{\pi}\int_{-\frac{\sqrt{C}}{\sqrt{n}}}^{\frac{\sqrt{C}}{\sqrt{n}}} Im(G_{\mu_n}(z))dx < \frac{2\sqrt{C}}{\pi \sqrt{n}}(\frac{C}{2n}+\sqrt{C})<\frac{2C}{3\sqrt{n}},
\end{equation}
as $C>5$ and $n>16C^2$.
From this estimation, ($\ref{eq.int.media.1}$), and ($\ref{eq.int.media.2}$), the desired result follows.

\end{proof}

Now, we are ready to conclude the proof.
Let $\epsilon_1=\frac{C}{\sqrt{n}}$ and $\epsilon_2=\frac{2}{\sqrt{n}}$.
By Theorem 1, we have that
\begin{align}
    \mu_n((-\infty,-1-\epsilon_1-\epsilon_2])&\leq -\frac{1}{\pi}\int_{-\infty}^{-1-\epsilon_1}Im(G_{\mu_n}(z))dx
    +\frac{2y}{\pi \epsilon_2}, \nonumber \\  
    \mu_n([-1+\epsilon_1+\epsilon_2,1-\epsilon_1-\epsilon_2])&\leq -\frac{1}{\pi}\int_{-1+\epsilon_1}^{1-\epsilon_1}Im(G_{\mu_n}(z))dx
    +\frac{2y}{\pi \epsilon_2}, \text{ and} \nonumber \\
    \mu_n([1+\epsilon_1+\epsilon_2,\infty])&\leq -\frac{1}{\pi}\int_{1+\epsilon_1}^{\infty}Im(G_{\mu_n}(z))dx
    +\frac{2y}{\pi \epsilon_2}. \nonumber
\end{align}

The previous lemmas implies that $\mu_n((-\infty,-1-\epsilon_1-\epsilon_2])<\frac{1}{\pi \sqrt{n}}+\frac{1}{\pi \sqrt{n}}$, $\mu_n([-1+\epsilon_1+\epsilon_2,1-\epsilon_1-\epsilon_2])<\frac{2C}{3 \sqrt{n}}+\frac{6}{\pi \sqrt{n}}+\frac{1}{\pi \sqrt{n}}$, and $\mu_n([1+\epsilon_1+\epsilon_2,\infty])<\frac{1}{\pi \sqrt{n}}+\frac{1}{\pi \sqrt{n}}$. Since $\frac{2C}{3 \sqrt{n}}+\frac{11}{\pi \sqrt{n}}<\frac{C}{\sqrt{n}}+\frac{2}{\sqrt{n}}=\epsilon_1+\epsilon_2$ for $C>5$, then we obtain that
\begin{equation}\nonumber
    \mu_n((-1-\epsilon_1-\epsilon_2,-1+\epsilon_1+\epsilon_2)\cup (1-\epsilon_1-\epsilon_2,1+\epsilon_1+\epsilon_2))> 1-\epsilon_1-\epsilon_2.
\end{equation}

Therefore, by Proposition ($\ref{ldmuab}$), we conclude that
\begin{equation}\nonumber
    L(\mu_n,\mathbf{b})\leq \frac{7}{2}\frac{C+2}{\sqrt{n}}.
\end{equation}

\section*{Aknowlwdgments}
Thanks to the PRODEP postdoc program of the UASLP.

\bibliographystyle{unsrt}




\end{document}